\documentclass[11pt,a4paper, envcountsame,mathserif]{amsart}
\usepackage[usenames,dvipsnames]{color}

\usepackage[utf8]{inputenc}	
\usepackage{pdfpages,url}

\usepackage{amsmath, amssymb,amsfonts,amsthm, stmaryrd, tikz, tikz-cd}
\usepackage[pagebackref=true,colorlinks]{hyperref}
\hypersetup{
    colorlinks=true,
    linkcolor=blue,
    citecolor=blue,
    }
\usepackage{thmtools}
\usepackage[all,curve]{xy}
\usepackage{thm-restate}
\usepackage{cleveref}

\theoremstyle{definition}
\newtheorem{theorem}{Theorem}

\newtheorem{definition}[theorem]{Definition}

\newtheorem{lemma}[theorem]{Lemma}

\newtheorem{corollary}[theorem]{Corollary}
\DeclareMathOperator{\Low}{Low}

\newcommand{\Ra}{\Rightarrow}

\newcommand{\Lra}{\Leftrightarrow}
\newcommand{\lowpi}{\Low(\Pi^0_1\text{-IM})}

\begin{document}
 \title{A Note on Lowness for $\Pi^0_1$-Immunity} 
\author[D.J. Webb]{David J. Webb}
\address{Chaminade University of Honolulu, 3140 Waialae Avenue
Honolulu, Hawaii 96816}
 \date{\today}
\begin{abstract} We show that the only degree which cannot co-enumerate a non-trivial $\Pi^0_1$-immune real is $\emptyset$, resolving a conjecture of the author. The main ingredient is that if a real $A$ has that all $A$-maximal sets are c.e., then $A$ has c.e.\ degree, which can be proved via a coding mechanism and two classical results of Lachlan.
\end{abstract} 
\pagenumbering{gobble}
\maketitle

We use standard notation from computability theory: see for instance \cite{Soare2016}.

\begin{definition}
A real $X\in 2^\omega$ is \emph{$\Pi^0_1$-immune} iff $X$ has no infinite $\Pi^0_1$\ subsets. Similarly $X$ is \emph{$\Pi^0_1(A)$-immune} iff $X$ has no infinite $\Pi^0_1(A)$\ subsets.
\end{definition}

\begin{definition}
A real $X$ is \emph{low for $\Pi^0_1$-immunity} iff any real $Y$ that is $\Pi^0_1$-immune is also $\Pi^0_1(X)$-immune. The class of such reals is $\lowpi$.
\end{definition}

It is a nice exercise to show that $\Low$(IM), the corresponding notion for standard immunity, is in fact the class of $\Delta^0_1$ reals.

The class $\lowpi$ was investigated in \cite{DJW}, though in a  different guise. There, the notion was named\footnote{Rather poorly, if I'm being honest.} $\Pi^0_1$BN$\Pi^0_1$IM, to denote those reals that do not co-enumerate a non-trivial $\Pi^0_1$-immune real. So to make use of this previous work, we must show these notions are equivalent:

\begin{lemma}\label{lowpichar} $\lowpi = \Pi^0_1$BN$\Pi^0_1$IM.\end{lemma}
\begin{proof}
    \textbf{[$\subseteq$]} If $X\in \lowpi$ and $Y\in\Pi^0_1(X)$ is infinite, then by definition $Y$ is not $\Pi^0_1(X)$-immune, and so by the lowness of $X$, not $\Pi^0_1$-immune either.\\
    \\
    \textbf{[$\supseteq$]} Let $X\in \Pi^0_1$BN$\Pi^0_1$IM, $Y$ be $\Pi^0_1$-immune, and $Z$ be a $\Pi^0_1(X)$ subset of $Y$. As $Z$ inherits $\Pi^0_1$-immunity from $Y$ and is co-enumerated by $X$, it must be finite. As $Z$ was arbitrary, $Y$ is $\Pi^0_1(X)$-immune, as desired.
\end{proof}

Much of the work on $\lowpi$ in \cite{DJW} involves maximal sets. Recall:
\begin{definition} A coinfinite c.e.\ set $M$ is maximal iff for all indices $e$, if $M\subseteq W_e$, then $W_e\setminus M$ is finite or $\overline{W_e}$ is.
\end{definition}
When relativizing to $A$-maximality, many elegant results become somewhat less so. For instance:

\begin{lemma}\label{LachlanA}(Lachlan, Theorem 2 of \cite{Lachlan1966}) There are maximal sets $M, N$ which form a minimal pair, i.e. if $B\leq_T M, N$, then $B$ is computable.\end{lemma}

Relativizing this result gives the following:

\begin{lemma}\label{L1966a} For every real $A$, there are $A$-maximal $M, N$ which form a minimal pair relative to $A$, i.e. if $B\leq_T M\oplus A, N\oplus A$, then $B\leq_T A$.\end{lemma}

Fortunately, a coding trick can be used to find an $A$-maximal set with the same degree as $M\oplus A$:
\begin{theorem}\label{MO} (Implicit in \cite{MO}) If $M$ is $A$-maximal, then there is an\linebreak $A$-maximal $M_1\equiv_T M\oplus A$.
\end{theorem}
\begin{proof} Recall that the prefix set\footnote{This is also called a Dekker set. Note that we identify strings $\sigma\in2^{<\omega}$ with their codes $n\in\omega$ under a fixed computable bijection.} of $A$ is $S(A) = \{\sigma\in 2^{<\omega}\mid \sigma\prec A\}$,
and given $A$-maximal $M$, define $M_1 = M\oplus_{S(A)} \omega.$ That is, the $n$th $1$ of $S(A)$ is replaced with the $n$th bit of $M$, and $\overline{S(A)}$ is filled in with $1$s. This gives a useful characterization in terms of the principal function of $S(A)$: $$n\in M \Lra p_{S(A)}(n) \in M_1.$$
To see that $M_1$ is $A$-maximal, first notice that $M_1$ is coinfinite and $A$-c.e.\ by inspection. Now for any $e$, $$M_1\subseteq W^A_e \Ra M\subseteq \{n\mid p_{S(A)}(n)\in W_e^A\}.$$
Call this set $V_e^A$ (since it is $A$-c.e.), and suppose that indeed $M_1\subseteq W_e^A$. Since $\overline{S(A)}\subseteq M_1 \subseteq W^A_e$, any member of $W_e^A\setminus M_1$ can be written as some $p_{S(A)}(n)$, so that $W_e^A\setminus M_1$ corresponds exactly (under $p_{S(A)}$) to $V_e^A\setminus M$.

Similarly, as $\overline{W_e^A}\subseteq S(A)$, any $m$ missing from $W_e^A$ corresponds to an index not in $V_e^A$, so that $|\overline{W_e^A}| = |\overline{V_e^A}|$.

Now the $A$-maximality of $M$ is inherited by $M_1$: if $V_e^A\setminus M$ is finite, then $W_e^A\setminus M_1$ is, and if $\overline{V_e^A}$ is finite, then $\overline{W_e^A}$ is.

Onwards to the degree calculation: by construction $\overline{M_1}\subseteq S(A)$, so as any infinite collection of prefixes of $A$ computes $A$, $M_1\geq_T A$.

Thus $M_1\geq_T S(A)$, so that it can compute $p_{S(A)}$. Now for each $n$, $M_1$ can check whether it contains $p_{S(A)}(n)$, and thus whether $n\in M$. So $M_1\geq_T M$.

By definition $M_1\leq_T M\oplus A$, so we have $M_1 \equiv_T M\oplus A$ as desired.
\end{proof}
It is interesting to note that while the simpler choice $N=M\oplus_A\omega$ would still be $A$-maximal, the proof of degree equivalence would have a gap or fail entirely: Soare  showed in \cite{Soare1969} that there are sets $A$ which cannot be computed by any $B\subseteq A$ with $|A\setminus B| = \infty$.

A nice consequence downstream of $A\in \lowpi$ is that $A$-maximal sets are all c.e.\ (\cite{DJW}, Theorem 1 of Section 3.9.1). In fact this property alone is rather strong:
\begin{theorem}
    If every $A$-maximal set is c.e., then $A$ has c.e.\ degree.
\end{theorem}
\begin{proof}
Let $A$ satisfy the hypothesis, and use \Cref{L1966a} to obtain $A$-maximal $M$ and $N$ forming a minimal pair relative to $A$. By \Cref{MO}, we can obtain $A$-maximal $M_1$ and $N_1$ forming a minimal pair not merely relative to $A$, but \emph{above} $A$: $A\leq_T M_1, N_1$, and if $B\leq_T M_1, N_1$, then $B\leq_TA$. By hypothesis, $M_1$ and $N_1$ are c.e., so their greatest lower bound (i.e. $\deg_T(A)$) is a c.e. degree by Lemma 18 of \cite{Lachlan1966}.
\end{proof}
\begin{corollary}\label{ce-deg}
    If $A\in \lowpi$, then $A$ has c.e.\ degree.
\end{corollary}
\newpage
Finally, we can restate Conjecture 3.8.1 of \cite{DJW} in the new notation and settle it in the affirmative:
\begin{theorem} $\lowpi\ = \Delta^0_1$.\end{theorem}
\begin{proof}
\textbf{[$\supseteq$]} This is immediate.\\
\\
\textbf{[$\subseteq$]} Let $A\in \lowpi$. Theorem 3.8.2 of \cite{DJW} (together with \Cref{lowpichar}) allows us to conclude that every c.e.\ $B$ below $A$ is computable: otherwise $B$ could co-enumerate a non-trivial $\Pi^0_1$-immune real, so that $A$ could. As $A$ itself has c.e.\ degree by \Cref{ce-deg}, $A\in\Delta^0_1$.
\end{proof}
\section*{Addendum}
Before the publication of this note, I found a shorter proof! The necessary ingredients are that if $A$ is low for $\Pi^0_1$-immunity, then $A$ is $\Delta^0_2$ and $A$-maximality coincides with maximality \cite{DJW}.
\begin{theorem} $\lowpi\ = \Delta^0_1$.
\end{theorem}
\begin{proof}
    For the non-trivial direction, fix $A\in \lowpi$, and let $M_0, M_1$ be the maximal sets from \Cref{LachlanA}. Immediately these are $A$-maximal, and so $A$-dense simple, i.e.\ the principal functions $p_{\overline{M_i}}$ dominate every function $f\leq_T A$.  As $A$ is $\Delta^0_2$, there is a function $c_A\leq_T A$ such that every function which dominates $c_A$ computes $A$ (\cite{Soare2016}, Theorem 5.6.6). Thus $c_A$ is dominated by $p_{\overline{M_i}}$, so $A\leq_T p_{\overline{M_i}}\leq_T M_i$. Since the $M_i$ form a minimal pair, $A\in\Delta^0_1$.
\end{proof}
\section*{Acknowledgments and AI Disclosure}
I would like to acknowledge Bj\o rn Kjos-Hanssen for feedback on the (many!) early attempts at proving this conjecture true in the course of my dissertation, and Emanuel Hosu for more recent helpful conversations about this conjecture that reignited my interest in resolving it.

ChatGPT 6 Astra was used for literature search and copyediting --- in particular, it identified the 1966 paper of Lachlan whose lemmata were instrumental in both versions of the proof. I take full responsibility for the mathematical content in this note.

\bibliographystyle{plain}
\bibliography{ref}
\end{document}